\documentclass[12pt, oneside]{article}   	
\usepackage[margin=1.25in]{geometry}    	
\usepackage{graphicx}				
\usepackage{amssymb}
\usepackage{amsmath}
\usepackage{tikz-cd}
\usepackage[mathcal]{eucal}
\usepackage{amsthm}
\theoremstyle{definition}
\newtheorem{thm}{Theorem}
\theoremstyle{definition}
\newtheorem{dfn}[thm]{Definition}
\theoremstyle{definition}
\newtheorem*{conv}{Conventions}
\theoremstyle{definition}

\theoremstyle{definition}

\theoremstyle{definition}
\newtheorem{lem}[thm]{Lemma}
\theoremstyle{definition}

\theoremstyle{definition}

\numberwithin{thm}{section}

\newcommand{\Hom}{\textup{Hom}}

\newcommand{\isom}{\approx}

\renewcommand{\lim}{\textup{lim}}
\renewcommand{\phi}{\varphi}
\newcommand{\id}{\textup{id}}

\newcommand{\AMod}{\textup{$A$-Mod}}

\newcommand{\Ab}{\textup{Ab}}

\usepackage[colorlinks,linkcolor=blue,citecolor=blue,urlcolor=blue,
            bookmarks,bookmarksnumbered]{hyperref}

\newcommand{\Ext}{\textup{Ext}}

\newcommand{\inv}{\textup{inv}}

\newcommand{\McoTors}{\textup{$M$-coTors}}

\newcommand{\ExMod}{\textup{ExMod}}
\newcommand{\NTors}{\textup{$N$-Tors}}
\newcommand{\coker}{\textup{coker}}
\newcommand{\coExMod}{\textup{coExMod}}
\newcommand{\Tor}{\textup{Tor}}

\title{Torsor-cotorsor duality of Ext groups}

\author{Nicholas Mertes\footnote{n.mertes@umiami.edu}\\ \textit{Department of Mathematics,}\\\textit{University of Miami},\\\textit{Coral Gables, FL}}

\date{June, 2020}							

\begin{document}

\maketitle

\begin{abstract}
Let $A$ be a ring, and let $M$ and $N$ be $A$-modules. Then $N$ can be viewed as a group object in the category $\AMod/ M$ of $A$-modules over $M$ and $\Ext^1(M, N)$ can be interpreted as the set of isomorphism classes of $N$-torsors. Alternatively, $M$ can be viewed as a cogroup object in the category $N/\AMod$ of $A$-modules under $N$ and $\Ext^1(M, N)$ can be interpreted as the set of isomorphism classes of $M$-cotorsors.
\end{abstract}

\section{Introduction}

This paper is a continuation of the ideas introduced in~\cite{thickcotor}. However, the current paper is written in such a way that it can be read independently with only a modest knowledge of abstract algebra and category theory. Familiarity with~\cite{dumfoote} and~\cite{Weibel} should suffice. In~\cite{thickcotor}, we showed that first order thickenings of schemes can be identified with cotorsors for the cogroup objects associated with quasi-coherent sheaves. This theorem was used to uncover the fact that the first scheme-theoretic Andre-Quillen homology group can be interpreted as a set of isomorphism classes of cotorsors. In doing so, the idea was developed that cotorsors play a role in homology theories which is analogous to the role of torsors in cohomology theories.

Strictly speaking, $\Ext$ should be defined in terms of injective resolutions and viewed as a cohomology functor. Thus $\Ext^1$ should have an interpretation in terms of torsors. However, $\Ext$ also has a dual description in terms of projective resolutions which makes $\Ext$ appear as a homology functor. If the cohomology-homology duality of $\Ext$ is to be taken seriously, then $\Ext^1$ should also have a dual description in terms of cotorsors. In this paper, we will show that this duality indeed gives rise to a description of $\Ext^1$ in terms of both torsors and cotorsors.
\begin{conv}
We fix a ring $A$ which need not be commutative but has a multiplicative identity. We always use the word \textit{$A$-module} to refer to a left $A$-module. We write $\AMod$ for the category of $A$-modules and $A$-module homomorphisms. We fix $A$-modules $M$ and $N$. We write $\AMod/M$ for the category of $A$-modules over $M$, and we write $N/\AMod$ for the category of $A$-modules under $N$.
\end{conv}
Since torsors are somewhat more natural than cotorsors, we first discuss torsors. In Section~\ref{torsorsection}, we show that $N$ can be viewed as a group object in $\AMod/ M$ via the trivial extension $M\oplus N\to M$. Then $\Ext^1(M, N)$ can be interpreted as the set of isomorphism classes of $N$-torsors. In Section~\ref{cotorsection}, we take the dual perspective of cotorsors. We show that $M$ can be viewed as a cogroup object in $N/\AMod$ via the trivial coextension $N\to M\oplus N$. Then $\Ext^1(M, N)$ can be interpreted as the set of isomorphism classes of $M$-cotorsors.

\section{Torsors}\label{torsorsection}

The general ideas and arguments presented in this section parallel the discussion of square-zero extensions in~\cite{thickcotor}. However, a few simplifications have been made in the present discussion. Our goal in this section is to define the category of $N$-torsors, and show that this category is equivalent to the category of extensions of $M$ by $N$. We thus conclude that $\Ext^1(M, N)$ can be identified with the set of isomorphism classes of $N$-torsors. Since we will only ever consider extensions of $M$ by $N$, we will simply call these extensions.
\begin{dfn}
An \textit{extension} is an object $f: P\to M$ of $\AMod/M$ such that $f$ is surjective, together with an isomorphism of $A$-modules $\alpha_f: N\to \ker(f)$.
\end{dfn}
\begin{dfn}
Let $f: P\to M$ and $g: Q\to M$ be extensions. A \textit{morphism of extensions} from $f$ to $g$ is a morphism $h: f\to g$ in $\AMod/M$ such that, for every $n\in N$, $h(\alpha_f(n)) = \alpha_g(n)$. We write $\ExMod(M, N)$ for the category of extensions and morphisms of extensions.
\end{dfn}
We now want to equip $\pi_M: M\oplus N\to M$ with the structure of a group object in $\AMod/M$. In order to do this, we need a notion of products and a terminal object. Note that the identity $\id_M: M \to M$ is a terminal object in $\AMod/M$. Let $f: P\to M$ and $g: Q\to M$ be objects of $\AMod/M$. We write $P\times_M Q$ for the submodule of $P\oplus Q$ such that
\[
P\times_M Q = \{(p, q)\in P\oplus Q\,|\, f(p) = g(q)\}.
\]
We write $f\times g: P\times_M Q\to M$ for the $A$-module homomorphism such that, for each $(p, q)\in P\times_M Q$, $(f\times g)(p, q) = f(p) = g(q)$. Then $f\times g$ is a product of $f$ and $g$ in $\AMod/ M$.
\begin{dfn}
We equip $\pi_M: M\oplus N\to M$ with the structure of a group object in $\AMod/M$ as follows. We define $e_N: \id_M\to \pi_M$ such that, for each $m\in M$,
\[
e_N(m) = (m, 0).
\]
We define $+_N: \pi_M\times\pi_M\to\pi_M$ such that, for each $((m, n_1), (m, n_2))\in (M\oplus N)\times_M (M\oplus N)$,
\[
(m, n_1) +_N (m, n_2) = (m, n_1 + n_2).
\]
We define $\inv_N: \pi_M\to \pi_M$ such that, for each $(m, n)\in M\oplus N$,
\[
\inv_N(m, n) = (m, -n).
\]
\end{dfn}
It is straightforward to verify that $\pi_M$ together with the three morphisms $e_N$, $+_N$, and $\inv_N$ in $\AMod/M$ define a group object (in fact an abelian group object) in $\AMod/M$. Therefore, we can define $\pi_M$-torsors in $\AMod/M$. We will call these torsors $N$-torsors rather than $\pi_M$-torsors for convenience of notation.

An $N$-torsor is, roughly speaking, an object $f: P\to M$ of $\AMod/M$ together with a group action $\tau_f: \pi_M\times f\to f$ satisfying the usual properties of a torsor. Explicitly, the product $\pi_M\times f$ is the $A$-module homomorphism $\pi_M\times f: (M\oplus N)\times_M P\to M$ such that, for each $((m, n), p)\in (M\oplus N)\times_M P$, $(\pi_M\times f)((m, n), p) = m = f(p)$. We now present a simplified description of $\pi_M\times f$ in the following definition and lemma.
\begin{dfn}
Let $f: P\to M$ be an extension. We write $\beta_f: N\oplus P\to M$ for the object of $\AMod/M$ such that, for each $(n, p)\in N\oplus P$, $\beta_f(n, p) = f(p)$.
\end{dfn}
\begin{lem}
If $f: P\to M$ is an extension, then $\beta_f$ is isomorphic to $\pi_M\times f$ in $\AMod/M$.
\end{lem}
\begin{proof}
Let $f: P\to M$ be an extension. We define a function $\phi: N\oplus P\to (M\oplus N)\times_M P$ such that, for each $(n, p)\in N\oplus P$, $\phi(n, p) = ((f(p), n), p)$. Note that, for each $((m, n), p)\in (M\oplus N)\times_M P$, we have that $m = f(p)$. Then it is straightforward to show that $\phi: \beta_f\to \pi_M\times f$ is an isomorphism in $\AMod/M$.
\end{proof}
Using this lemma, we can identify the group action $\tau_f: \pi_M\times f\to f$ with the corresponding morphism $\tau_f: \beta_f\to f$ in $\AMod/M$. Thus $\tau_f$ can be viewed as an $A$-module homomorphism $\tau_f: N\oplus P\to P$, and hence $\tau_f$ has the appearance of an action of $N$ on $P$.
\begin{dfn}
An \textit{$N$-torsor} is an object $f: P\to M$ of $\AMod/M$ such that $f$ is surjective, together with a morphism $\tau_f: \beta_f\to f$ in $\AMod/M$ such that
\begin{enumerate}
\item For each $p\in P$, $\tau_f(0, p) = p$.
\item For each $(p_1, p_2)\in P\times_M P$, there exists a unique $n\in N$ such that
\[
\tau_f(n, p_2) = p_1.
\]
\end{enumerate}
\end{dfn}
One may suspect that we also need to require an associativity axiom. However, associativity is implied by our definition of $N$-torsor. We prove this associativity now, together with a preliminary lemma.
\begin{lem}
Let $f: P\to M$ be an $N$-torsor. If $(n, p)\in N\oplus P$, then
\[
\tau_f(n, p) = \tau_f(n, 0) + p.
\]
\end{lem}
\begin{proof}
Let $(n, p)\in N\oplus P$. Then
\[
\begin{split}
\tau_f(n, p) &= \tau_f((n, 0) + (0, p)) \\
&= \tau_f(n, 0) + \tau_f(0, p) \\
&= \tau_f(n, 0) + p.
\end{split}
\]
\end{proof}
\begin{lem}
Let $f: P\to M$ be an $N$-torsor. If $n_1, n_2\in N$ and $p\in P$, then
\[
\tau_f(n_1 + n_2, p) = \tau_f(n_1, \tau_f(n_2, p)).
\]
\end{lem}
\begin{proof}
Let $n_1, n_2\in N$ and let $p\in P$. Then
\[
\begin{split}
\tau_f(n_1 + n_2, p) &= \tau_f((n_1, 0) + (n_2, p)) \\
&= \tau_f(n_1, 0) + \tau_f(n_2, p) \\
&= \tau_f(n_1, \tau_f(n_2, p)).
\end{split}
\]
\end{proof}
We now show that each extension has a natural $N$-torsor structure.
\begin{dfn}
Let $f: P\to M$ be an extension. We define a function $\tau_f: N\oplus P\to P$ such that, for each $(n, p)\in N\oplus P$,
\[
\tau_f(n, p) = \alpha_f(n) + p.
\]
\end{dfn}
\begin{lem}
If $f: P\to M$ is an extension, then $\tau_f: N\oplus P\to P$ gives $f$ the structure of an $N$-torsor.
\end{lem}
\begin{proof}
Let $f: P\to M$ be an extension. We first need to verify that $\tau_f: \beta_f\to f$ is a morphism in $\AMod/M$. It is straightforward to show that $\tau_f$ is a homomorphism of $A$-modules. Furthermore, $\tau_f$ is a morphism in $\AMod/M$ since, for each $(n, p)\in N\oplus P$,
\[
\begin{split}
f(\tau_f(n, p)) &= f(\alpha_f(n) + p) \\
&= f(\alpha_f(n)) + f(p) \\
&= f(p) \\
&= \beta_f(n, p).
\end{split}
\]
It remains to show that the morphism $\tau_f$ in $\AMod/M$ indeed satisfies the properties of an $N$-torsor. Note that, for each $p\in P$,
\[
\begin{split}
\tau_f(0, p) &= \alpha_f(0) + p\\
&= p.
\end{split}
\]
Now let $(p_1, p_2)\in P\times_M P$. Then $p_1 - p_2\in \ker(f)$. Since $\alpha_f: N\to \ker(f)$ is bijective, there exists a unique $n\in N$ such that $\alpha_f(n) = p_1 - p_2$. Equivalently, there exists a unique $n\in N$ such that $\tau_f(n, p_2) = \alpha_f(n) + p_2 = p_1$.
\end{proof}
We now show that the assignment of each extension to its corresponding $N$-torsor gives rise to a functor. In particular, we define the notion of a morphism of $N$-torsors and we show that each morphism of extensions induces a corresponding morphism of $N$-torsors.
\begin{dfn}
Let $f: P\to M$ and $g: Q\to M$ be $N$-torsors. A \textit{morphism of $N$-torsors} from $f$ to $g$ is a morphism $h: f\to g$ in $\AMod/ M$ such that, for each $(n, p)\in N\oplus P$, $h(\tau_f(n, p)) = \tau_g(n, h(p))$. We write $\NTors$ for the category of $N$-torsors and morphisms of $N$-torsors.
\end{dfn}
\begin{lem}\label{morphisquarezeroextimptormorph}
Let $f: P\to M$ and $g: Q\to M$ be extensions. If $h: f\to g$ is a morphism of extensions, then $h$ is a morphism of $N$-torsors.
\end{lem}
\begin{proof}
Let $h: f\to g$ be a morphism of extensions. Let $(n, p)\in N\oplus P$. Then
\[
\begin{split}
h(\tau_f(n, p)) &= h(\alpha_f(n) + p) \\
&= h(\alpha_f(n)) + h(p) \\
&= \alpha_g(n) + h(p) \\
&= \tau_g(n, h(p)).
\end{split}
\]
\end{proof}
\begin{dfn}
We write $\Psi: \ExMod(M, N)\to \NTors$ for the functor which maps each extension to its corresponding $N$-torsor and which maps each morphism of extensions to its corresponding morphism of $N$-torsors.
\end{dfn}
\begin{thm}\label{extorstheorem}
The functor $\Psi: \ExMod(M, N)\to \NTors$ is an equivalence of categories.
\end{thm}
\begin{proof}
We will show that $\Psi$ is fully faithful and essentially surjective. Note that $\Psi$ is faithful since $\Psi$ maps each morphism of extensions to itself.

We now show that $\Psi$ is full. Let $f: P\to M$ and $g: Q\to M$ be extensions and let $h: f\to g$ be a morphism of $N$-torsors. We want to show that $h$ is a morphism of extensions. Let $n\in N$. Then
\[
\begin{split}
h(\alpha_f(n)) &= h(\tau_f(n, 0)) \\
&= \tau_g(n, h(0)) \\
&= \tau_g(n, 0) \\
&= \alpha_g(n)
\end{split}
\]
and thus $h$ is a morphism of extensions.

Finally, we show that $\Psi$ is essentially surjective. Let $f: P\to M$ be an $N$-torsor. We want to show that $f$ can be given the structure of an extension in such a way that the $N$-torsor structure associated with the extension $f$ is the same as the original $N$-torsor structure on $f$. Note that, for each $n\in N$,
\[
\begin{split}
f(\tau_f(n, 0)) &= \beta_f(n, 0) \\
&= f(0) \\
&= 0
\end{split}
\]
and thus $\tau_f(n, 0)\in \ker(f)$. We define a function $\alpha_f: N\to \ker(f)$ such that, for each $n\in N$, $\alpha_f(n) = \tau_f(n, 0)$. Then $\alpha_f$ is bijective since, for each $p\in \ker(f)$, $(p, 0)\in P\times_M P$ and thus there exists a unique $n\in N$ such that $\alpha_f(n) = \tau_f(n, 0) = p$.

We now show that $\alpha_f: N\to \ker(f)$ is an isomorphism of $A$-modules. Let $n_1, n_2\in N$. Then
\[
\begin{split}
\alpha_f(n_1 + n_2) &= \tau_f(n_1 + n_2, 0) \\
&= \tau_f(n_1, \tau_f(n_2,0)) \\
&= \tau_f(n_1, 0) + \tau_f(n_2, 0) \\
&= \alpha_f(n_1) + \alpha_f(n_2)
\end{split}
\]
and thus $\alpha_f$ is a group homomorphism. Now let $a\in A$ and let $n\in N$. Then
\[
\begin{split}
\alpha_f(a n) &= \tau_f(a n, 0) \\
&= \tau_f(a(n, 0)) \\
&= a \tau_f(n, 0) \\
&= a \alpha_f(n)
\end{split}
\]
and thus $\alpha_f$ preserves $A$-scalar multiplication. Therefore, $\alpha_f: N\to \ker(f)$ is an isomorphism of $A$-modules and hence $f$ together with $\alpha_f$ is an extension.

It remains to be shown that the $N$-torsor structure associated with the extension $f$ is the same as the original $N$-torsor structure on $f$. Let $(n, p)\in N\oplus P$. Then
\[
\begin{split}
\tau_f(n, p) &= \tau_f(n, 0) + p \\
&= \alpha_f(n) + p.
\end{split}
\]
Therefore, the functor $\Psi: \ExMod(M, N)\to \NTors$ is an equivalence of categories.
\end{proof}
Recall that $\Ext^1(M, N)$ can be identified with the set of isomorphism classes of objects of $\ExMod(M, N)$. Therefore, by the above equivalence of categories, $\Ext^1(M, N)$ can also be identified with the set of isomorphism classes of $N$-torsors. In Section~\ref{cotorsection}, we will show that $\Ext^1(M, N)$ has a dual description as the set of isomorphism classes of $M$-cotorsors.

\section{Cotorsors}\label{cotorsection}

In effort to make this section easier to read, an attempt has been made to follow the organization of section~\ref{torsorsection} as closely as possible. Our goal in this section is to define the category of $M$-cotorsors, and show that this category is equivalent to the category of coextensions of $N$ by $M$. We will then show that the category of coextensions of $N$ by $M$ is equivalent to the category of extensions of $M$ by $N$, and thus we conclude that $\Ext^1(M, N)$ has a dual description as the set of isomorphism classes of $M$-cotorsors. Since we will only ever consider coextensions of $N$ by $M$, we will simply call these coextensions.
\begin{dfn}
A \textit{coextension} is an object $f: N\to P$ of $N/\AMod$ such that $f$ is injective, together with an isomorphism of $A$-modules $\alpha_f: \coker(f)\to M$.
\end{dfn}
Let $f: N\to P$ and $g: N\to Q$ be coextensions. Let $h: f\to g$ be a morphism in $N/\AMod$. We will show that $h: \coker(f)\to \coker(g)$ is a well-defined $A$-module homomorphism. Let $p\in P$ and let $n\in N$. Then
\[
\begin{split}
h(p + f(n)) &= h(p) + h(f(n)) \\
&= h(p) + g(n),
\end{split}
\]
and hence $h$ is well-defined on equivalence classes. If $p\in P$, then we write $[p]$ for the equivalence class of $p$ in $\coker(f)$.
\begin{dfn}
Let $f: N\to P$ and $g: N\to Q$ be coextensions. A \textit{morphism of coextensions} from $f$ to $g$ is a morphism $h: f\to g$ in $N/\AMod$ such that, for every $[p]\in \coker(f)$, $\alpha_g(h([p])) = \alpha_f([p])$. We write $\coExMod(N, M)$ for the category of coextensions and morphisms of coextensions.
\end{dfn}
We now want to equip $i_N: N\to M\oplus N$ with the structure of a cogroup object in $N/\AMod$. In order to do this, we need a notion of coproducts and an initial object. Note that the identity $\id_N: N \to N$ is an initial object in $N/\AMod$. Let $f: N\to P$ and $g: N\to Q$ be objects of $N/\AMod$. We define
\[
P\amalg_N Q = \coker((f, -g): N\to P\oplus Q).
\]
We define $f\amalg g: N \to P\amalg_N Q$ as the $A$-module homomorphism such that, for each $n\in N$, $(f\amalg g)(n) = (f(n), 0) = (0, g(n))$. Then $f\amalg g$ is a coproduct of $f$ and $g$ in $N/\AMod$.

In order to figure out the counit, coaddition, and coinversion morphisms associated with $i_N$, we will use the perspective of corepresentable functors. To give $i_N$ the structure of an (abelian) cogroup, we want to extend the corepresentable functor associated with $i_N$ to a functor
\[
\Hom_{N/\AMod}(i_N, -): N/\AMod\to\Ab
\]
where $\Ab$ denotes the category of abelian groups. Let $f: N\to P$ be an object of $N/\AMod$. We want to give $\Hom_{N/\AMod}(i_N, f)$ the structure of an abelian group. Let $g\in \Hom_{N/\AMod}(i_N, f)$ and let $(m, n)\in M\oplus N$. Then
\[
\begin{split}
g(m, n) &= g((m, 0) + (0, n)) \\
&= g(m, 0) + g(0, n) \\
&= g(m, 0) + g(i_N(n)) \\
&= g(m, 0) + f(n). \\
\end{split}
\]
Let $h\in \Hom_{N/\AMod}(i_N, f)$. We define $g + h\in \Hom_{N/\AMod}(i_N, f)$ such that
\[
(g + h)(m, n) = g(m, 0) + h(m, 0) + f(n).
\]
This addition defines an abelian group structure on $\Hom_{N/\AMod}(i_N, f)$. Furthermore, this abelian group structure is functorial in $f$.

We now use the functor $\Hom_{N/\AMod}(i_N, -): N/\AMod\to\Ab$ to equip $i_N$ with the structure of a cogroup object in $N/\AMod$. First note that, for each object $f$ of $N/\AMod$, $\Hom_{N/\AMod}(i_N, f)$ is an abelian group and thus has an identity element. The assignment of the unique morphism $\id_N\to f$ in $N/\AMod$ to the identity element in $\Hom_{N/\AMod}(i_N, f)$ induces a natural transformation $\Hom_{N/\AMod}(\id_N, -)\to \Hom_{N/\AMod}(i_N, - )$ which corresponds, via the coYoneda lemma, to the counit $e_M: i_N\to \id_N$. Similarly, the addition operation on $\Hom_{N/\AMod}(i_N, f)$ induces a natural transformation
\[
\Hom_{N/\AMod}(i_N, - )\times \Hom_{N/\AMod}(i_N, - )\to \Hom_{N/\AMod}(i_N, - ).
\]
Using the universal property of the coproduct, we have that
\[
\Hom_{N/\AMod}(i_N, - )\times \Hom_{N/\AMod}(i_N, - )\isom \Hom_{N/\AMod}(i_N\amalg i_N, - )
\]
and thus we obtain a natural transformation
\[
\Hom_{N/\AMod}(i_N\amalg i_N, - )\to \Hom_{N/\AMod}(i_N, - ).
\]
corresponding to the coaddition $+_M: i_N\to i_N \amalg i_N$. Finally, the inversion operation on $\Hom_{N/\AMod}(i_N, f)$ induces a natural transformation
\[
\Hom_{N/\AMod}(i_N, - )\to \Hom_{N/\AMod}(i_N, - )
\]
corresponding to the coinversion $\inv_M: i_N\to i_N$.

Using the above discussion, we have the following explicit description of the cogroup structure on $i_N$.
\begin{dfn}
We equip $i_N: N\to M\oplus N$ with the structure of a cogroup object in $N/\AMod$ as follows. We define $e_M: i_N\to \id_N$ such that, for each $(m, n)\in M\oplus N$,
\[
e_M(m, n) = n.
\]
We define $+_M: i_N\to i_N \amalg i_N$ such that, for each $(m, n)\in M\oplus N$,
\[
\begin{split}
+_M (m, n) &= ((m, n), (m, 0)) \\
&= ((m, 0), (m, n)).
\end{split}
\]
We define $\inv_M: i_N\to i_N$ such that, for each $(m, n)\in M\oplus N$,
\[
\inv_M(m, n) = (-m, n).
\]
\end{dfn}
Given that $i_N$ is a cogroup object in $N/\AMod$, we can define $i_N$-cotorsors in $N/\AMod$. We will call these cotorsors $M$-cotorsors rather than $i_N$-cotorsors for convenience of notation.

An $M$-cotorsor is, roughly speaking, an object $f: N\to P$ of $N/\AMod$ together with a cogroup action $\tau_f: f\to i_N\amalg f$ satisfying properties dual to those of a torsor. Explicitly, the coproduct $i_N\amalg f$ is the $A$-module homomorphism $i_N\amalg f: N\to (M\oplus N)\amalg_N P$ such that, for each $n\in N$, $(i_N\amalg f)(n) = ((0, n), 0) = ((0, 0), f(n))$. We now present a simplified description of $i_N\amalg f$ in the following definition and lemma.
\begin{dfn}
Let $f: N\to P$ be a coextension. We write $\beta_f: N\to M\oplus P$ for the object of $N/\AMod$ such that, for each $n\in N$, $\beta_f(n) = (0, f(n))$.
\end{dfn}
\begin{lem}
If $f: N\to P$ is a coextension, then $\beta_f$ is isomorphic to $i_N\amalg f$ in $N/\AMod$.
\end{lem}
\begin{proof}
Let $f: N\to P$ be a coextension. We define a function $\phi: M\oplus P\to (M\oplus N)\amalg_N P$ such that, for each $(m, p)\in M\oplus P$, $\phi(m, p) = ((m, 0), p)$. It is straightforward to show that $\phi: \beta_f\to i_N \amalg f$ is a morphism in $N/\AMod$. It remains to show that $\phi$ is bijective.

We first show that $\phi$ is injective. Let $(m_1, p_1), (m_2, p_2)\in M\oplus P$ be such that $\phi(m_1, p_1) = \phi(m_2, p_2)$. Then $((m_1, 0), p_1) = ((m_2, 0), p_2)\in (M\oplus N)\amalg_N P$. Thus there exits $n\in N$ such that $((m_1 - m_2, 0), p_1 - p_2) = ((0, n), -f(n))\in (M\oplus N)\oplus P$. Therefore, $(m_1, p_1) = (m_2, p_2)$.

We now show that $\phi$ is surjective. Let $((m, n), p)\in (M\oplus N)\amalg_N P$. Then, since $((0, n), 0) = ((0, 0), f(n))\in (M\oplus N)\amalg_N P$, we have that
\[
\begin{split}
((m, n), p) &= ((m, 0) + (0, n), p) \\
&= ((m, 0), p) + ((0, n), 0) \\
&= ((m, 0), p) + ((0, 0), f(n)) \\
&= ((m, 0), p + f(n)).
\end{split}
\]
Then we see that $\phi(m, p + f(n)) = ((m, 0), p + f(n)) = ((m, n), p)$.
\end{proof}
Using this lemma, we can identify the cogroup action $\tau_f: f\to i_N \amalg f$ with the corresponding morphism $\tau_f: f\to \beta_f$ in $N/\AMod$. Thus $\tau_f$ can be viewed as an $A$-module homomorphism $\tau_f: P\to M\oplus P$, and hence $\tau_f$ has the appearance of a coaction of $M$ on $P$.
\begin{dfn}
Let $f: N\to P$ be an object of $N/\AMod$. We write $\pi_P: \beta_f\to f$ for the morphism in $N/\AMod$ such that, for each $(m, p)\in M\oplus P$, $\pi_P(m, p) = p$.
\end{dfn}
\begin{dfn}
An \textit{$M$-cotorsor} is an object $f: N\to P$ of $N/\AMod$ such that $f$ is injective, together with a morphism $\tau_f: f\to \beta_f$ in $N/\AMod$ such that
\begin{enumerate}
\item For each $p\in P$, $\pi_P(\tau_f(p)) = p$.
\item For each $(m, p)\in M\oplus P$, there exists a unique $(p_1, p_2)\in P\amalg_N P$ such that
\[
\tau_f(p_1) + (0, p_2) = (m, p).
\]
\end{enumerate}
\end{dfn}
One comment needs to be made about this definition of $M$-cotorsor. Let $f: N\to P$ be an object of $N/\AMod$ and let $\tau_f: f\to \beta_f$ be a morphism in $N/\AMod$. In order for property (2.) to be well-defined, we need to show that $\tau_f(p_1) + (0, p_2)$ is independent of the choice of representative for the equivalence class $(p_1, p_2)\in P\amalg_N P$. Suppose that $(p_1, p_2) = (p_1', p_2')\in P\amalg_N P$. Then there exists $n\in N$ such that $(p_1 - p_1', p_2 - p_2') = (f(n), -f(n))$. Therefore,
\[
\begin{split}
\tau_f(p_1 - p_1') &= \tau_f(f(n)) \\
&= \beta_f(n) \\
&= (0, f(n)) \\
&= (0, p_2' - p_2)
\end{split}
\]
and thus we have that
\[
\tau_f(p_1) + (0, p_2) = \tau_f(p_1') + (0, p_2').
\]

We now show that each coextension has a natural $M$-cotorsor structure.
\begin{dfn}
Let $f: N\to P$ be a coextension. We define a function $\tau_f: P\to M\oplus P$ such that, for all $p\in P$,
\[
\tau_f(p) = (\alpha_f([p]), p).
\]
\end{dfn}
\begin{lem}
If $f: N\to P$ is a coextension, then $\tau_f: P\to M\oplus P$ gives $f$ the structure of an $M$-cotorsor.
\end{lem}
\begin{proof}
It is straightforward to show that $\tau_f$ is a homomorphism of $A$-modules. Now note that $\tau_f: f\to \beta_f$ is a morphism in $N/\AMod$ since, for every $n\in N$,
\[
\begin{split}
\tau_f(f(n)) &= (\alpha_f([f(n)]), f(n)) \\
&= (\alpha_f(0), f(n)) \\
&= (0, f(n)) \\
&= \beta_f(n).
\end{split}
\]
It remains to show that the morphism $\tau_f$ in $N/\AMod$ indeed satisfies the properties of an $M$-cotorsor. Note that, for each $p\in P$,
\[
\begin{split}
\pi_P(\tau_f(p)) &= \pi_P(\alpha_f([p]), p) \\
&= p.
\end{split}
\]
Now let $(m, p)\in M\oplus P$. We want to show that there exists a unique $(p_1, p_2)\in P\amalg_N P$ such that
\[
\tau_f(p_1) + (0, p_2) = (m, p).
\]
Since $\alpha_f: \coker(f)\to M$ is surjective, choose $[p_1]\in \coker(f)$ such that $\alpha_f([p_1]) = m$. Now define $p_2 = p - p_1$. Then $(p_1, p_2)\in P\amalg_N P$ is such that
\[
\begin{split}
\tau_f(p_1) + (0, p_2) &= (\alpha_f([p_1]), p_1) + (0, p_2) \\
&= (m, p_1) + (0, p_2) \\
&= (m, p_1 + p_2) \\
&= (m, p).
\end{split}
\]
Now suppose that $(p_1', p_2')\in P\amalg_N P$ is such that
\[
\begin{split}
\tau_f(p_1') + (0, p_2') &= (\alpha([p_1']), p_1' + p_2') \\
&= (m, p).
\end{split}
\]
Then, since $\alpha_f$ is injective, we have that $[p_1] = [p_1']$ and hence there exists $n\in N$ such that $p_1 - p_1' = f(n)$. Furthermore, $p_1 + p_2 = p_1' + p_2'$ and thus $p_1 - p_1' = -(p_2 - p_2') = f(n)$. Thus $(p_1 - p_1', p_2- p_2') = (f(n), -f(n))$, and therefore $(p_1, p_2) = (p_1', p_2')\in P\amalg_N P$.
\end{proof}
We now show that the assignment of each coextension to its corresponding $M$-cotorsor gives rise to a functor. In particular, we define the notion of a morphism of $M$-cotorsors and we show that each morphism of coextensions induces a corresponding morphism of $M$-cotorsors.
\begin{dfn}
Let $f: N\to P$ and $g: N\to Q$ be $M$-cotorsors. A \textit{morphism of $M$-cotorsors} from $f$ to $g$ is a morphism $h: f\to g$ in $N/\AMod$ such that the diagram
\[
\begin{tikzcd}
P \arrow{r}{\tau_f} \arrow[swap]{d}{h} & M\oplus P \arrow{d}{\id_M\oplus h} \\
Q \arrow{r}{\tau_g} & M\oplus Q
\end{tikzcd}
\]
commutes. We write $\McoTors$ for the category of $M$-cotorsors and morphisms of $M$-cotorsors.
\end{dfn}
\begin{lem}\label{morphisquarezeroextimptormorph}
Let $f: N\to P$ and $g: N\to Q$ be coextensions. If $h: f\to g$ is a morphism of coextensions, then $h$ is a morphism of $M$-cotorsors.
\end{lem}
\begin{proof}
Let $h: f\to g$ be a morphism of coextensions. Let $p\in P$. Then
\[
\begin{split}
\tau_g(h(p)) &= (\alpha_g([h(p)]), h(p)) \\
&= (\alpha_g(h([p])), h(p)) \\
&= (\alpha_f([p]), h(p)) \\
&= (\id_M\oplus h)(\alpha_f([p]), p) \\
&= (\id_M\oplus h)(\tau_f(p)).
\end{split}
\]
\end{proof}
\begin{dfn}
We write $\Phi: \coExMod(N, M)\to \McoTors$ for the functor which maps each coextension to its corresponding $M$-cotorsor and which maps each morphism of coextensions to its corresponding morphism of $M$-cotorsors.
\end{dfn}
\begin{thm}\label{extorstheorem}
The functor $\Phi: \coExMod(N, M)\to \McoTors$ is an equivalence of categories.
\end{thm}
\begin{proof}
We will show that $\Phi$ is fully faithful and essentially surjective. Note that $\Phi$ is faithful since $\Phi$ maps each morphism of coextensions to itself.

We now show that $\Phi$ is full. Let $f: N\to P$ and $g: N\to Q$ be coextensions and let $h: f\to g$ be a morphism of $M$-cotorsors. We want to show that $h$ is a morphism of coextensions. Let $[p]\in \coker(f)$. Then
\[
\begin{split}
(\alpha_g(h([p])), h(p)) &=  \tau_g(h(p)) \\
&=  (\id_M\oplus h)(\tau_f(p)) \\
&= (\id_M\oplus h)(\alpha_f([p]), p) \\
&= (\alpha_f([p]), h(p))
\end{split}
\]
and thus $\alpha_g(h([p])) = \alpha_f([p])$.

Finally, we show that $\Phi$ is essentially surjective. Let $f: N\to P$ be an $M$-cotorsor. We want to show that $f$ can be given the structure of a coextension in such a way that the $M$-cotorsor structure associated with the coextension $f$ is the same as the original $M$-cotorsor structure on $f$.

Let $\pi_M: M\oplus P\to M$ be the canonical projection. Define $\alpha_f: \coker(f)\to M$ such that, for each $[p]\in\coker(f)$, $\alpha_f([p]) = \pi_M(\tau_f(p))$. We need to check that $\alpha_f$ is well-defined. Let $n\in N$. Then
\[
\begin{split}
\pi_M(\tau_f(p + f(n))) &= \pi_M(\tau_f(p) + \tau_f(f(n))) \\
&= \pi_M(\tau_f(p) + \beta_f(n)) \\
&= \pi_M(\tau_f(p) + (0, f(n))) \\
&= \pi_M(\tau_f(p)) + \pi_M(0, f(n)) \\
&= \pi_M(\tau_f(p)).
\end{split}
\]
Since $\tau_f$ and $\pi_M$ are $A$-module homomorphisms, we see that $\alpha_f: \coker(f)\to M$ is an $A$-module homomorphism. We now show that $\alpha_f$ is bijective. Let $m\in M$. Since $f: N\to P$ is an $M$-cotorsor, there exists a unique $(p_1, p_2)\in P\amalg_N P$ such that
\[
\tau_f(p_1) + (0, p_2) = (m, 0).
\]
Then $[p_1]$ is such that
\[
\begin{split}
\alpha_f([p_1]) &= \pi_M(\tau_f(p_1)) \\
&= \pi_M(m, -p_2) \\
&= m,
\end{split}
\]
so $\alpha_f$ is surjective. Now assume $p_1'\in P$ is such that $\alpha_f([p_1']) = m$. Since $\pi_P(\tau_f(p_1')) = p_1'$, we have that
\[
\begin{split}
\tau_f(p_1') + (0, -p_1') &= (m, p_1') + (0, -p_1') \\
&= (m, 0) \\
\end{split}
\]
and hence $(p_1, p_2) = (p_1', -p_1')\in P\amalg_N P$. Therefore, there exists $n\in N$ such that $(p_1 - p_1', p_2 + p_1') = (f(n), - f(n))$. Then $p_1 = p_1' + f(n)$ and hence $[p_1] = [p_1']$ so $\alpha_f$ is injective. Therefore, $\alpha_f: \coker(f)\to M$ is an isomorphism of $A$-modules and hence $f$ together with $\alpha_f$ is a coextension.

It remains to be shown that the $M$-cotorsor structure associated with the coextension $f$ is the same as the original $M$-cotorsor structure on $f$. Let $p\in P$. Then
\[
\begin{split}
\tau_f(p) &= (\pi_M(\tau_f(p)), \pi_P(\tau_f(p))) \\
&= (\alpha_f([p]), p) \\
\end{split}
\]
Therefore, the functor $\Phi: \coExMod(N, M)\to \McoTors$ is an equivalence of categories.
\end{proof}
Our final step is to show that $\coExMod(N, M)$ is equivalent to $\ExMod(M, N)$. This will conclude the proof that the set of isomorphism classes of $M$-cotorsors can be identified with $\Ext^1(M, N)$. We first show that each extension $f: P\to M$ induces a corresponding coextension $f': N\to P$.
\begin{dfn}
Let $f: P\to M$ be an extension. Then we write $f': N\to P$ for the $A$-module homomorphism such that, for each $n\in N$, $f'(n) = \alpha_f(n)$. We write $\alpha_{f'}: \coker(f')\to M$ for the isomorphism of $A$-modules such that, for each $[p]\in \coker(f')$, $\alpha_{f'}([p]) = f(p)$.
\end{dfn}
\begin{lem}
Let $f: P\to M$ and $g: Q\to M$ be extensions. If $h: f\to g$ is a morphism of extensions, then $h: f'\to g'$ is a morphism of coextensions.
\end{lem}
\begin{proof}
Let $h: f\to g$ be a morphism of extensions. Then $h:P\to Q$ is an $A$-module homomorphism such that $g\circ h = f$ and, for each $n\in N$, $h(\alpha_f(n)) = \alpha_g(n)$. Let $n\in N$. Then
\[
\begin{split}
h(f'(n)) &= h(\alpha_f(n)) \\
&= \alpha_g(n) \\
&= g'(n),
\end{split}
\]
and thus $h$ is a morphism in $N/\AMod$. Now let $[p]\in\coker(f')$. Then
\[
\begin{split}
\alpha_{g'}(h([p])) &= \alpha_{g'}([h(p)]) \\
&= g(h(p)) \\
&= f(p) \\
&= \alpha_{f'}([p]),
\end{split}
\]
and therefore $h: f'\to g'$ is a morphism of coextensions.
\end{proof}
\begin{dfn}
We write $\Theta: \ExMod(M, N)\to \coExMod(N, M)$ for the functor which maps each extension to its corresponding coextension and which maps each morphism of extensions to its corresponding morphism of coextensions.
\end{dfn}
\begin{thm}
The functor $\Theta: \ExMod(M, N)\to \coExMod(N, M)$ is an equivalence of categories.
\end{thm}
\begin{proof}
We will show that $\Theta$ is fully faithful and essentially surjective. Note that $\Theta$ is faithful since $\Theta$ maps each morphism of extensions to itself.

We now show that $\Theta$ is full. Let $f: P\to M$ and $g: Q\to M$ be extensions and let $h: f'\to g'$ be a morphism of coextensions. We want to show that $h: f\to g$ is a morphism of extensions. Let $p\in P$. Then
\[
\begin{split}
g(h(p)) &= \alpha_{g'}([h(p)]) \\
&= \alpha_{g'}(h([p])) \\
&= \alpha_{f'}([p]) \\
&= f(p)
\end{split}
\]
and thus $h$ is a morphism in $\AMod/M$. Now let $n\in N$. Then
\[
\begin{split}
h(\alpha_f(n)) &= h(f'(n)) \\
&= g'(n) \\
&= \alpha_g(n),
\end{split}
\]
and therefore $h: f\to g$ is a morphism of extensions.

Finally, we show that $\Theta$ is essentially surjective. Let $f: N\to P$ be a coextension. We will show that there exists an extension $g: P\to M$ such that $g' = f$. We define $g: P\to M$ such that, for each $p\in P$, $g(p) = \alpha_f([p])$. We define $\alpha_g: N\to \ker(g)$ such that, for each $n\in N$, $\alpha_g(n) = f(n)$. If $n\in N$, then
\[
\begin{split}
g(\alpha_g(n)) &= g(f(n)) \\
&= \alpha_f([f(n)]) \\
&= \alpha_f(0) \\
&= 0
\end{split}
\]
so indeed we have that $\alpha_g(n)\in\ker(g)$. It is straightforward to show that $g' = f$. Therefore, the functor $\Theta: \ExMod(M, N)\to \coExMod(N, M)$ is an equivalence of categories.
\end{proof}

\section*{Conclusion}

In this paper, we have shown that the cohomology-homology duality of $\Ext$ gives rise to a dual description of $\Ext^1(M, N)$ in terms of both torsors and cotorsors. Moving forward, it is of interest to see which other algebraic homology theories have convenient descriptions in terms of cotorsors. In particular, it would be satisfying to find an interpretation of $\Tor_1(M, N)$ in terms of cotorsors. It was shown in~\cite{thickcotor} that ring-theoretic Andre-Quillen cohomology has an interpreation in terms of torsors. It thus seems reasonable to suspect that ring-theoretic Andre-Quillen homology has an interpretation in terms of cotorsors.


\begin{thebibliography}\raggedright



\bibitem{thickcotor}
First order thickenings and cotorsors, Nicholas Mertes. arXiv:2006.04230. \\
\url{https://arxiv.org/abs/2006.04230}

\bibitem{dumfoote}
Abstract algebra, David Dummit and Richard Foote.

\bibitem{Weibel}
An Introduction to Homological Algebra, Charles Weibel.

\end{thebibliography}
\end{document}